\documentclass[]{amsart}
\usepackage{amscd,amsthm,amssymb,amsfonts,amsmath,euscript}

\usepackage{comment}

\usepackage{tikz-cd}
\usepackage{tikz}
\usetikzlibrary{matrix,arrows}
\tikzset{cong/.style={draw=none,edge node={node [sloped, allow upside down, auto=false]{$\cong$}}},
         Isom/.style={above,every to/.append style={edge node={node [sloped, allow upside down, auto=false]{$\sim$}}}}}

\theoremstyle{plain}
\newtheorem{thm}{Theorem}
\newtheorem{lemma}[thm]{Lemma}
\newtheorem{prop}[thm]{Proposition}
\newtheorem{cor}[thm]{Corollary}

\theoremstyle{definition}
\newtheorem{defn}[thm]{Definition}
\newtheorem{eg}[thm]{Example}

\theoremstyle{remark}
\newtheorem{remark}[thm]{Remark}

\newcommand{\nc}{\newcommand}

\def\makeop#1{\expandafter\def\csname#1\endcsname
  {\mathop{\rm #1}\nolimits}\ignorespaces}
\makeop{Hom}   \makeop{End}   \makeop{Aut}   \makeop{Isom}  \makeop{Pic} 
\makeop{Gal}   \makeop{ord}   \makeop{Char}  \makeop{Div}   \makeop{Lie} 
\makeop{PGL}   \makeop{Corr}  \makeop{PSL}   \makeop{sgn}   \makeop{Spf}
\makeop{Spec}  \makeop{Tr}    \makeop{Nr}    \makeop{Fr}    \makeop{disc}
\makeop{Proj}  \makeop{supp}  \makeop{ker}   \makeop{im}    \makeop{dom}
\makeop{coker} \makeop{Stab}  \makeop{SO}    \makeop{SL}    \makeop{SL}
\makeop{Cl}    \makeop{cond}  \makeop{Br}    \makeop{inv}   \makeop{rank}
\makeop{id}    \makeop{Fil}   \makeop{Frac}  \makeop{GL}    \makeop{SU}
\makeop{Nrd}   \makeop{Sp}    \makeop{Tr}    \makeop{Trd}   \makeop{diag}
\makeop{Res}   \makeop{ind}   \makeop{depth} \makeop{Tr}    \makeop{st}
\makeop{Ad}    \makeop{Int}   \makeop{tr}    \makeop{Sym}   \makeop{can}
\makeop{length}\makeop{SO}    \makeop{torsion} \makeop{GSp} \makeop{Ker}
\makeop{Adm}   \makeop{Mat}
\def\makebb#1{\expandafter\def
  \csname bb#1\endcsname{{\mathbb{#1}}}\ignorespaces}
\def\makebf#1{\expandafter\def\csname bf#1\endcsname{{\bf
      #1}}\ignorespaces} 
\def\makegr#1{\expandafter\def
  \csname gr#1\endcsname{{\mathfrak{#1}}}\ignorespaces}
\def\makescr#1{\expandafter\def
  \csname scr#1\endcsname{{\EuScript{#1}}}\ignorespaces}
\def\makecal#1{\expandafter\def\csname cal#1\endcsname{{\mathcal
      #1}}\ignorespaces} 

\def\doLetters#1{#1A #1B #1C #1D #1E #1F #1G #1H #1I #1J #1K #1L #1M
                 #1N #1O #1P #1Q #1R #1S #1T #1U #1V #1W #1X #1Y #1Z}
\def\doletters#1{#1a #1b #1c #1d #1e #1f #1g #1h #1i #1j #1k #1l #1m
                 #1n #1o #1p #1q #1r #1s #1t #1u #1v #1w #1x #1y #1z}
\doLetters\makebb   \doLetters\makecal  \doLetters\makebf
\doLetters\makescr 
\doletters\makebf   \doLetters\makegr   \doletters\makegr

     \def\qed{\qedmark\medbreak}%
\def\qedmark{{\enspace\vrule height 6pt width 5pt depth 1.5pt}}%
    \def\setminus{\smallsetminus}
\def\Gm{{{\bbG}_{\rm m}}}   \def\Ga{{{\bbG}_{\rm a}}}

\normalsize

\makeop{Bl}

\def\Spec{{\rm Spec}\,}

\def\Fpbar{\overline{\bbF}_p}
\def\Fp{{\bbF}_p}
\def\Fq{{\bbF}_q}

\def\Qp{{\bbQ}_p}

\def\Zp{{\bbZ}_p}

\newcommand{\Z}{\mathbb Z}
\newcommand{\Q}{\mathbb Q}
\newcommand{\R}{\mathbb R}

\newcommand{\A}{\mathbb A}    
\newcommand{\F}{\mathbb F}

\newcommand{\npr}{\noindent }

\newcommand{\<}{\langle}   
\renewcommand{\>}{\rangle} 

\newcommand{\isoto}{\stackrel{\sim}{\longrightarrow}}
\nc{\embed}{\hookrightarrow}

\newcommand{\ac}{algebraically closed }
\newcommand{\dieu}{Dieudonn\'{e} }

\nc{\ol}{\overline}
\nc{\wt}{\widetilde}
\nc{\opp}{\mathrm{opp}}
\def\ul{\underline}
\def\onto{\twoheadrightarrow}
\def\der{{\rm der}}
\def\wh{\widehat}

\makeop{Ram}
\makeop{Rep}

\begin{document}
\renewcommand{\thefootnote}{\fnsymbol{footnote}}
\setcounter{footnote}{-1}
\numberwithin{equation}{section}


\title[Dieudonn\'e modules and supersingular abelian varieties]{Introduction to Dieudonn\'e modules and supersingular abelian varieties revisited}
\author{Chia-Fu Yu}
\address{
Institute of Mathematics, Academia Sinica \\
Astronomy Mathematics Building \\
No.~1, Roosevelt Rd. Sec.~4 \\ 
Taipei, Taiwan, 10617} 
\email{chiafu@math.sinica.edu.tw}





\maketitle

\begin{abstract}
In this Note we present an expository account of \dieu modules and
revisit supersingular abelian varieties. We give a simple proof of
the uniqueness of products of two or more supersingular elliptic
curves (a theorem due to Deligne, Ogus and Shioda), and of Oort's theorem for superspecial abelian varieties. 
\end{abstract} 

 
\section{Introduction}
The purpose of this Note is two-fold. One is to introduce basic theory for \dieu modules. The other is to explain their arithmetic applications to supersingular abelian varieties in more detail. 
Although this note contains essentially no new results (except Lemma~\ref{lm:15} and its applications); some of the  proofs presented here given differ from the original proofs and are elementary. 
This Note may serve as a supplement to the article by Harashita~\cite{harashita:2024RIMS} in this volume and Karemaker's lecture notes for the 2024 Arizona Winter School~\cite{karemaker:2024AWS}. 

Let $p$ be a prime number. 
In this Note, all ground fields are of characteristic 
$p$, unless specified otherwise. 

Recall that an elliptic curve $E$ over a field $K$ is called \emph{supersingular} if $E[p](\ol K)=0$. Let $\Lambda$ be the set of isomorphism classes of supersingular elliptic curves over $\Fpbar$. 
Let $E\in \Lambda$, $B:=\End^0(E)$ be the endomorphism algebra of $E$, which is a quaternion algebra over $\Q$ ramified precisely at $p$ and $\infty$, and $O_B:=\End(E)$ be the endomorphism ring of $E$. 
The  Deuring--Eichler correspondence states that there is a bijection 
\begin{equation}
    \label{eq:1.1}
    \Lambda\simeq \Cl (O_B),
\end{equation}
where $\Cl (O_B)$ denotes the set of right $O_B$-ideal classes, and one has the adelic interpretation:
\[ \Cl (O_B)\simeq B^\times \backslash (B\otimes \A_f)^\times/ \wh O_B^\times. \]
Here $\A_f$ is the finite adele ring of $\Q$. The well-known class number formula then shows
\begin{equation}
    \label{eq:1.2}
   |\Lambda|=\frac{p-1}{12}+\frac{1}{3} \left(1-\left (\frac{-3}{p} \right ) \right)+\frac{1}{4} \left(1-\left (\frac{-4}{p} \right ) \right), 
\end{equation}
where $(\cdot/p)$ denotes the Kronecker symbol. 

We shall explore a higher dimensional generalization of \eqref{eq:1.1} and \eqref{eq:1.2} (see Remark~\ref{rem:corr}), and \dieu theory comes into play.  

\section{\dieu modules}

\subsection{} 
Throughout this paper, $k$ denotes a perfect field of characteristic $p>0$. Let $W:=W(k)$ denote the ring of Witt vectors over
$k$, and $L=W[1/p]$ the fraction field of $W$. 
Let $\sigma$ be the Frobenius map on $W$ and
$L$ induced by the endomorphism $\sigma:k \to k$, $a\mapsto a^p$ for $a\in k$. We shall assume $k$ to be algebraically closed in Lemma~\ref{lm:elliptic}, Theorem~\ref{thm:5} and Subsection~\ref{sec:2.4}.

If $H$ is a commutative group or a sheaf of commutative groups on a scheme $S$, and $m$ is an integer, denote by $H[m]$ the kernel of the multiplication map $m: H\to H$.

\begin{defn}\label{def:DM}

{\rm (1)} A \emph{\dieu module} over $k$ is a finite $W$-module $M$ together with a $\sigma$-linear map $F$ (Frobenius) and a
$\sigma^{-1}$-linear map $V$ (Verschiebung) on $M$ such that $F V=VF=p$.

{\rm (2)} (cf.~\cite[p.~45--46]{demazure}) A \emph{$p$-divisible group} over $k$ is a direct limit $G=\varinjlim G_n$ of commutative finite locally free group schemes $G_n$ over $k$ such that 
\begin{itemize}
    \item [(i)] $G_n \isoto G_m[p^n]$ for all integers $m\ge n\ge 0$; 
    \item[(ii)] For any $n\ge 0$, the multiplication map $p: G_{n+1} \to G_n$ is faithfully flat (i.e.~flat and surjective). 
\end{itemize}  
\end{defn}

Let $A_k=W(k)[F,V]$ denote the non-commutative ring modulo the relations:
\[ FV=VF=p, \quad F a = a^\sigma F, \quad a V=V a^{\sigma}, \quad  a\in W. \]
Then a \dieu module $M$ is a finitely generated $A_k$-module that is a finite $W$-module. If $G=\varinjlim G_n$ is a $p$-divisible group, then there exists an integer $h\ge 0$, called the \emph{height} of $G$ and denoted by $h(G)$, such that $\rank G_n=p^{nh}$ for all $n\ge 0$. \\

\begin{eg}\label{eg:2} \ 

\begin{enumerate}
    \item $M_1=\<e_1,e_2\>$ is a free $W$-module of rank two with relations 
\[ F(e_1)=e_1, \quad V(e_1)=pe_1, \quad F(e_2)=p e_2, \quad V(e_2)=e_2. \]
 \item $M_2=\<e_1,e_2\>$ is a free $W$-module of rank two with relations 
\[ F(e_1)=V(e_1)=e_2, \quad F(e_2)=V(e_2)=pe_1. \]
 \item If $X$ is an abelian variety, then $X[p^\infty]:=\varinjlim X[p^n]$ is a $p$-divisible group of height $2 \dim X$.\\  
\end{enumerate}
\end{eg}

Classical \dieu theory is stated as follows.

\begin{thm}[{cf.~\cite[p.~69--70]{demazure}}]\label{thm:dieu}
    There is an
anti-equivalence of categories between the category of $p$-divisible
groups over $k$ and the category of $W$-free \dieu modules over $k$. 
\end{thm}

For simplicity, \dieu modules will refer as 
$W$-free \dieu modules. The \dieu module associated to a $p$-divisible group $G$ is denoted by $M(G)$. 
The contravariant \dieu functor $G\mapsto M(G)$ is exact and commutes with any perfect field base change and the dual. The dual of $M$ is defined by 
\[ M^t:=\Hom_{W}(M,W), \quad F(f)(x):=f(V (x))^\sigma, \quad V(f)(x):=f(F(x))^{\sigma^{-1}},\]
for any $f\in M^t$ and $x\in M$. Then there is a natural isomorphism between $M(G^t)$ and the dual $M^t$ of $M=M(G)$, where $G^t$ is the Serre dual of $G$ \cite[Sect.~1.4]{li-oort}.
Moreover, we have $h(G)=\rank_W M$ and there is a natural isomorphism $\Lie(G)\simeq \Hom_k(M/FM, k)$.

The \dieu module $M(E):=M(E[p^\infty])$ associated to an elliptic curve $E$ over $k$ is a \dieu module $M$ of rank $2$ satisfying $\dim_k M/F M=\dim_k M/V M=1$.

\begin{lemma}\label{lm:elliptic}
    Assume that $k$ is algebraically closed. Then $M=M(E)$ is either isomorphic to  $M_1$ or isomorphic to $M_2$ in Example~\ref{eg:2}. Moreover, $E$ is ordinary (resp.~supersingular) if $M\simeq M_1$ (resp.~$M\simeq M_2$). 
\end{lemma}
\begin{proof}
    Put $\ol M:=M/pM$. Then we have either (a) $F\ol M \neq V \ol M$, or (b) $F\ol M = V \ol M$. In case (a), the maps $F: F\ol M\to F\ol M$ and $V: V\ol M \to V \ol M$ are bijective. Then there exist elements $e_1, e_2\in M\setminus pM$ such that $F e_1\equiv e_1 \mod pM$ and $Ve_2\equiv e_2 \mod pM$. 
     We now show that $F^\infty M:=\cap_{m\ge 1} F^m M$ is a free $W$-module of rank $1$. Consider the short exact sequence 
    \[ 0 \to F^m M/p^m M \to M / p^m M\to M/F^m M\to 0, \]
    and $F^m M/p^m M$ and $M/F^m M$ are $W$-modules of length $m$. Since $F^m e_1 \not \in pM$ for all $m\ge 1$, we see that both are free $W/p^mW$-modules of rank $1$. This shows that $F^\infty M$ is a free $W$-module of rank $1$. Since the map $F: F^\infty M\to F^\infty M$ is bijective, there exists a generator $e_1\in F^\infty M$ such that $F e_1=e_1$. Similarly, $V^\infty M:= \cap_{m\ge 1} V^m M$ is a free $W$-module of rank one and there exists a generator $e_2\in V^\infty M$ such that $Ve_2=e_2$ and therefore $M\simeq M_1$.
    Since $F$ is isomorphic to $F^\infty M$, the \dieu submodule $F^\infty M$ is \'etale, that is, it corresponds to the \'etale $p$-divisible group $E[p^\infty]^e$ of height one over $k$. This shows that $E[p](k)=\Z/p\Z$ and $E$ is ordinary. In case (b), we have $F^2 M=pM$. The skeleton $M^\diamond:=\{m\in M : F^2m=p m\}$ is a \dieu module over $\F_{p^2}$ and we have the isomorphism $M^\diamond \otimes_{W(\F_{p^2})} W(k)\isoto M$; see \cite[Sect.~5.7]{li-oort}. Taking an element $e_1\in M^\diamond \setminus (F,V) M^\diamond$ and letting $e_2:=Fe_1$,  then we have $M\simeq M_2$.
     Since there is no \'etale part of the $p$-divisible group $E[p^\infty]$, $E$ is supersingular. \qed 
\end{proof}


\subsection{Newton polygons and $a$-numbers}
To any relatively prime nonnegative integers $a$ and
$b$ with $(a,b)\neq (0,0)$, one associates a \dieu module over $\Fp$
\[ M_{a,b}:=W(\Fp)[F,V]/(F^a-V^b). \]
The rational number $\lambda=b/(a+b)$ is called the slope of $M_{a,b}$. 

\begin{thm}[{Manin-Dieudonn\'e \cite[Chap.~II, ``Classification Theorem'',
  p.~35]{manin:thesis}}]\label{thm:5}
Assume that $k$ is algebraically closed. For every \dieu module $M$ over $k$, there exist relatively prime pairs $(a_i,b_i)\neq (0,0)$ of
nonnegative integers and integers $m_i\ge 1$ for $i=1,\dots ,r$ such that
there is an isomorphism of rational \dieu modules
\begin{equation}
  \label{eq:2.1}
M\otimes_{W(k)} L \simeq \bigoplus_{i=1}^r M_{a_i,b_i}^{m_i} \otimes_{W(\Fp)} L.  
\end{equation} 
\end{thm}

We shall give a proof of the Manin-\dieu Theorem in Section~\ref{sec:3}. 
The \emph{slope sequence} of $M$ is defined to be the tuple $(\lambda_1^{\oplus
  (a_1+b_1)m_1},\dots ,\lambda_r^{\oplus
  (a_r+b_r)m_r})$ ordered so that $\lambda_i<\lambda_{i+1}$ for all $i$,
  where $\lambda_i:=b_i/(a_i+b_i)$, and the rational numbers $\lambda_1,\dots, \lambda_r$
  are called the \emph{slopes} of $M$. When $k$ is an arbitrary perfect field, the slope sequence and slopes of a \dieu module $M$ over $k$ are defined to be those of $M\otimes_{W(k)} W(\bar k)$, where $\bar k$ is an algebraic closure of $k$.
  

The \emph{$a$-number} of $M$ is defined by \[ a(M):=\dim_k M/(F,V)M. \] 
One has $a(M\otimes W(k'))=a(M)$ for any perfect field extension $k'/k$. One also has $a(M^t)=a(M)$.
Clearly, $a(M_{1,1}^{\oplus g})=g$. Conversely, if $M$ is
a \dieu module with $\dim M/VM=\dim M/FM=a(M)=g$ and $k$ is algebraically closed, then we have $M\simeq M_{1,1}^{\oplus g}\otimes W(k)$. The proof is the same as that in Lemma~\ref{lm:elliptic}.

For any abelian variety $X$ over $k$, the contravariant \dieu module associated to $X$ is defined by $M(X):=M(X[p^\infty])$.
The slope sequence and slopes of an abelian variety $X$ over a field $K$ are defined as those of $M(X\otimes_K \ol K)$. The $a$-number of $X$ is defined by 
\begin{equation}\label{eq:2.2}
a(X):=\dim_{\ol K} \Hom(\alpha_p \otimes \ol K, X\otimes \ol K),  
\end{equation}
where $\alpha_p:=\Spec \Fp[t]/(t^p)$ is the kernel of the Frobenius morphism $\Ga\to \Ga$, $x \mapsto x^p$ of the additive group $\Ga$ over $\Fp$. Then $a(X)=a(M(X\otimes \ol K))$.

\begin{remark}
By \cite[Proposition 2.5]{li-oort}, for any commutative finite group scheme $G$ over a field $K$ there exists a smallest subgroup scheme $A(G)$ of $G$ containing all of the images $\alpha_p \to G$ over $K$. We give an example showing that $ A(G)\otimes_K K'\subsetneq A(G\otimes_K K')$ may occur for a field extension $K'/K$.

Consider the  commutative finite group scheme $G$ over a non-perfect field $K$ such that its Cartier dual $G^D=\Spec A$ is a subgroup scheme of $(\Ga)^2$, where $A=K[x,y]/(x^{p^2}, y^p-\lambda x^p)$ with $\lambda\in K \setminus K^p$. Let $K':=K[\mu]$, where $\mu^p=\lambda$. As any map $g: \alpha_{p,K} \to G$ induces a map $g^D: G^D=\Spec A \to \alpha_{p,K}^D=\alpha_{p,K}=\Spec K[t]/(t^p)$, the element $f=(g^D)^*(t)$ satisfies $m^*(f)=f\otimes 1+1\otimes f$ and $f^p=0$, where $m^*: A\to A\otimes_K A$ is the co-multiplication of $A$. Therefore any map $(g^D)^*: K[t]/(t^p)\to A$ factors through the subalgebra $K[T]\subseteq A$, where 
\[ T:=\{f\in A: m^*(f)=f\otimes 1+1\otimes f, f^p=0\}.\]
Alternatively, the surjective map $g^D: G^D \to \alpha_{p,K}$ factors through $G^D \onto H$, where $H=\Spec K[T]$. Then $A(G):=H^D\subseteq G$ is the smallest subgroup scheme containing all of the images $\alpha_p \to G$ over $K$.
To show $ A(G)\otimes_K K'\subsetneq A(G\otimes_K K')$, it suffices to show $T\otimes_K K' \subsetneq T'$, where 
\[ T':=\{f\in A\otimes_K K': m_{K'}^*(f)=f\otimes 1+1\otimes f, f^p=0\}.\]
One computes that
\[ \begin{split}
    T &=\{f\in A: f \text{ is an additive function and } f^p=0\} \\
      &=\{f\in \<1, x, x^p,y, y^p=\lambda x^p \>_K: f^p=0\}=\<x^p\>_K.
\end{split} \]
Similarly, we have 
\[ T'
   =\{f\in \<1, x, x^p,y \>_{K'}: f^p=0\}=\<x^p, y-\mu x \>_{K'}. \]
Thus, $T\otimes_K K' \subsetneq T'$ and $A(G)\otimes_K K'\subsetneq A(G\otimes_K K')$.  

Observe that $\dim_K \Hom_K(\alpha_{p,K},G)=\dim T$. We have then
\[ \dim_K \Hom_K(\alpha_{p,K},G) < \dim_{K'} \Hom_{K'} (\alpha_{p,K'},G\otimes_K K'). \]
This explains why one defines $a(X)$ by \eqref{eq:2.2}.
\end{remark}

\subsection{Cartier-\dieu theory} \label{sec:2.3}
We briefly introduce Cartier-\dieu theory. The theory classifies commutative smooth formal groups over a general base and provides a framework for deformation theory of $p$-divisible smooth formal groups in terms of covariant \dieu modules. 
We follow the convenient setting of \cite[Section 0]{norman:algo} and 
\cite[Section 2, p.~1011]{chai-norman:gamma2}.

\newcommand{\Cart}{{\mathrm{Cart}_p}}

Let $R$ be a commutative ring of characteristic $p$.
Let $W(R)$ denote the ring of
Witt vectors over $R$, equipped with the Verschiebung $\tau$ and
Frobenius $\sigma$:
\begin{equation*}
  \begin{split}
    (a_0,a_1,\dots)^\tau&=(0, a_0, a_1,\dots) \\
    (a_0,a_1,\dots)^\sigma&=(a_0^p, a_1^p,\dots).
  \end{split}
\end{equation*}
Let $\Cart(R)$ denote the Cartier ring $W(R)[F][\![V]\!]$ modulo the
relations
\begin{itemize}
\item[$\bullet$] $FV=p$ and $VaF=a^\tau$,
\item[$\bullet$] $Fa=a^\sigma F$ and $Va^\tau=aV, \ \forall\, a\in W(R).$
\end{itemize}

\begin{defn}
A left $\Cart(R)$-module is \emph{uniform} if it is complete and
  separated in the $V$-adic topology. A uniform $\Cart(R)$-module $M$ is
  \emph{reduced} if $V$ is injective on $M$ and $M/VM$ is a finite free
  $R$-module. A (covariant) \emph{\dieu module over $R$} is a finitely generated
  reduced uniform $\Cart(R)$-module.
\end{defn}

\begin{defn}
    {\rm (1)} Let $n\in \Z_{\ge 0}$. An $n$-dimensional \emph{commutative smooth formal group} over $R$ is a smooth formal scheme $G$ over $R$ of dimension $n$ together with a commutative formal group law ($m:G\wh\otimes_R G\to G$ and $\varepsilon: \Spec R \to G$). More precisely, $(G,m)\simeq (\Spf R[\ul X], F(\ul X,\ul Y))$, where $\ul X=(X_1,\dots , X_n)$, $\ul Y=(Y_1,\dots , Y_n)$ and 
    \[ F(\ul X,\ul Y)=(F_1(\ul X, \ul Y), \dots, F_n(\ul X,\ul Y))\in R[\ul X,\ul Y]^n \] 
    such that 
    \begin{itemize}
        \item [(i)] $F(\ul X,\ul Y)=\ul X+\ul Y +$ terms of higher degree;
        \item [(ii)] $F(\ul X, F(\ul Y,\ul Z))=F(F(\ul X, \ul Y), \ul Z)$;
        \item[(iii)] $F(\ul Y,\ul X)=F(\ul X, \ul Y)$. 
    \end{itemize}
    It follows from (i) and (ii) that $F(\ul X,0)=F(0,\ul X)=\ul X$ and there exists a unique power series $\psi(\ul X)\in R[\ul X]$ such that $F(\psi(\ul X), X)=F(X,\psi(\ul X))=0$.

    {\rm (2)} A \emph{$p$-divisible smooth formal group} over $R$ is a commutative smooth formal group over $R$ such that 
    \begin{itemize}
        \item [(i)] $G=\bigcup_n G[p^n]$;
        \item [(ii)] the multiplication map $p:G\to G$ is surjective  
    \end{itemize}
    as fppf sheaves on the category of Artinian $R$-algebras.
\end{defn}
The definition of $p$-divisible groups over a general base $R$ of characteristic $p$ is analogous to Definition~\ref{def:DM}(2). A connected $p$-divisible group can also be regarded as a $p$-divisible smooth formal group.

\begin{thm}[{\cite{zink:cartier}, \cite[Theorem 3.3]{chai:cartier}}]\label{thm:cartier} 
   There is an equivalence of categories between the category of finite-dimensional commutative smooth formal groups over $R$ and the category of
covariant \dieu modules over $R$.  
\end{thm}

We denote the covariant \dieu functor by $\bfD_*$. The tangent space of a commutative smooth formal group $G$ is canonically isomorphic to
$\bfD_*(G)/V\bfD_*(G)$.

If $G$ is a connected $p$-divisible group over $k$, then the covariant \dieu module $\bfD_*(G)$ is a torsion-free \dieu module $M$ over $k$ 
(Definition~\ref{def:DM}) such that $V^m M\subseteq pM$ for some $m\ge 1$. Moreover,
there is a canonical isomorphism $M(G)^t\simeq \bfD_*(G)$ of \dieu modules. 
Let $G^{(p)}$ denote the pull-back of $G/k$ by the absolute Frobenius map $F_{\Spec k}: \Spec k \to \Spec k$.
The relative Frobenius morphism $F_{G/k}: G\to G^{(p)}$ induces the operator $F$ on $M(G)$ and the operator $V$ on $\bfD_*(G)$, respectively. The relative Verschiebung morphism $V_{G/k}: G^{(p)}\to G$ induces the operator $V$ on $M(G)$ and the operator $F$ on $\bfD_*(G)$, respectively. Thus, the  map $V_{G/k}: \Lie(G^{(p)}) \to \Lie(G)$ induces the $\sigma$-linear map $F$ on $\bfD_*(G)/V \bfD_*(G)$.

The covariant functor $\bfD_*$ applies only to the smooth part of a $p$-divisible group. In order to carry the information of the \'etale part,     
the covariant \dieu module of an abelian variety $X$ over $k$ is defined by 
\[ D_*(X):=\bfD_*(X[p^\infty]^0)\oplus M(X[p^\infty]^{e})^t, \]
where $X[p^\infty]^0$ and $X[p^\infty]^{e}$ are respective connected and \'etale parts of the $p$-divisible group $X[p^\infty]$. 

\subsection{The construction of deformations} \label{sec:2.4} 
Let $M_0$ be the covariant \dieu module of a $p$-divisible smooth formal group $G_0$ over an \ac field $k$. Let $g:=\dim_k M_0/V M_0$ and $h=\dim_k (V M_0)/p M_0$.
Choose a $W$-basis $e_1,\dots , e_{g+h}$ of $M_0$ such that
\begin{equation}
\begin{split}
    F e_i &=\sum_{j=1}^{g+h} a_{ij} e_j, \quad  \quad (i=1,\dots, g), \\ 
    e_i &= V(\sum_{j=1}^{g+h} a_{ij} e_j),  \quad (i=g + 1, \dots, g + h),  
\end{split}
\end{equation}
for some invertible matrix $(a_{ij})$ with entries in $W$.
 
 \begin{thm}[{\cite[Theorem 1]{norman:algo}}]\label{thm:defo}
    Let $(R,\grm)$ be a local Artinian $k$-algebra with residue field $k$.  
    There is a one-to-one correspondence between isomorphism classes of deformations of $M_0$ over $R$ and maps 
    \[ \begin{split}
        d: VM_0/M_0 & \longrightarrow \grm \otimes_k  (M_0/ V M_0),\\ \ol e_i & \mapsto \sum_{j=1}^g \bar d_{ij} \otimes \ol e_j, \quad (g+1\le i \le g+h),
    \end{split} \]
    where $\ol e_i$, $\ol e_j$ are the images of $e_i, e_j$ in $VM_0/p M_0$ and $M_0/ V M_0$ respectively. Set
 $d_{ij} = (\bar d_{ij}, 0, ...) \in W(R)$ for $g+1\le i \le g+h$ and $1\le j \le g$, and $d_{ij} = 0$ otherwise. The \dieu module $M_d$ corresponding to the map $d$ is defined by the generators $e_1,\dots, e_{g+h}$ and relations
\begin{equation}\label{eq:2.4}
    \begin{split}
    F e_i &=\sum_{j=1}^{g+h} a_{ij} \left (e_j+\sum_{k=1}^g d_{jk} e_k \right ), \quad  \quad (i=1,\dots, g), \\ 
    e_i &= V\left (\sum_{j=1}^{g+h} a_{ij} \left (e_j+\sum_{k=1}^g d_{jk} e_k \right ) \right ),  \quad (i=g + 1, \dots, g + h).  
\end{split}
\end{equation} 
 Furthermore, every \dieu module over $R$ which reduces to $M_0$ over $k$ is isomorphic to exactly one \dieu module $M_d$. 
 \end{thm}

\section{Proof of the Manin-\dieu Theorem}\label{sec:3}

We provide a proof of the Manin-\dieu Theorem for the reader's convenience, as this is a very important result in \dieu theory. We follow the reference Demazure~\cite{demazure}. 
In this section, we let $k$ be a perfect field of characteristic $p$, and assume it to be algebraically closed in Subsection~\ref{sec:4.3}.

\subsection{$F$-spaces}

Recall that $L=W[1/p]$ denotes the fraction field of the ring $W=W(k)$ of Witt vectors over $k$. 
\begin{defn}

    (1) An \emph{$F$-space over $k$} is a finite-dimensional $L$-vector space $N$ together with a bijective $\sigma$-linear map $F:N\isoto N$. 

    (2) An \emph{$F$-lattice over $k$} is a finite free $W$-module $M$ together with an injective $\sigma$-linear map $F:M\to M$. 

    (3) An $F$-space $N$ over $k$ is called \emph{effective} if it contains an $F$-lattice of full rank. 
\end{defn}
An $F$-lattice $M$ is a \dieu module if and only if $pM \subseteq FM$. 
For any $\lambda\in \Q$, write $\lambda=s/r$ with $r\ge 1,s\in \Z$ and $(r,s)=1$. Define
\[ N^\lambda:=\Qp[T]/(T^r-p^s)\]
and let $F$ be the multiplication by $T$. Then $(N^\lambda, F)$ is an $F$-space over $\Fp$. Define
\[ N^\lambda_k:=L\otimes_{\Qp} N^\lambda, \quad F(a\otimes x):=\sigma(a)\otimes F(x), \quad  a\in L,\ x\in N^\lambda.\]
Then $(N^\lambda_k,F)$ is an $F$-space over $k$. Put $e_i:=T^{i-1}$, then we have an $L$-basis $e_1,\dots, e_{r}$ of $N^\lambda_k$ with relations
\[ F(e_1)=e_2, \quad F(e_2)=e_3, \ \dots \  , F(e_{r-1})=e_r, \quad F(e_r)=p^s e_1. \]
If $\lambda\ge 0$ (i.e.~$s\ge 0$), we put
\[ M^\lambda:=\Zp[T]/(T^r-p^s), \quad M^\lambda_k:=W\otimes_{\Zp} M^\lambda,\]
and then $(M^\lambda_k, F)$ is an $F$-lattice over $k$ and the $F$-space $N^\lambda_k$ is effective.

Set $\xi=p^{1/r}$ and $L[\xi]=L[p^{1/r}]=L[X]/(X^r-p)$, which is a field extension of $L$ of degree $r$. Regard $L[\xi]$ as an $L$-vector space and define
$F_s(\xi^i)=\xi^{i+s}$ for all $i$.
Then $(L[\xi], F_s)$ is an $F$-space over $k$. Put $e_i:=\xi^i$, for $i\in \Z$, then $L[\xi]$ is generated by $e_i$ with relations
\[ e_{i+r}=pe_i, \quad F_s(e_i)=e_{i+s}, \quad  i\in\Z. \]
For example, if $r=5$ and $s=2$, then $F_s$ acts as
\[ e_0 \longmapsto e_2 \longmapsto  e_4 \longmapsto \  e_6=p e_1 \  \longmapsto  \ e_8=pe_3 \ \longmapsto  \ e_{10}=p^2 e_0.\]
Observe that if $0\le s\le r$, then $(W[\xi], F_s)$ is a \dieu module, while the $F$-lattice $(M^\lambda_k, F)$ is not a \dieu module except when $s=0,1$. There is an isomorphism 
\[ N_k^\lambda \simeq (L[\xi], F_s), \quad \{1,T,\dots, T^{r-1} \} \longmapsto \{e_0, e_s, \dots, e_{s(r-1)} \}. \]

\npr {\it Notation.} 
For two $F$-spaces $(N,F), (N',F')$ over $k$, denote \[ N\otimes N':=(N\otimes_{L} N', F''), \] 
where $F''(x\otimes x')=F(x)\otimes F'(x')$ for $x\in N, x'\in N'$. 
For any $n\in \Z$, denote $\bbO(n):=(L\cdot e, F(e)=p^{-n} e)$. 
For an $F$-space $N$ over $k$, denote \[ N(n):=N\otimes \bbO(n)=(N, F_n), \quad F_n:=p^{-n}F. \] 
Clearly, $N(-n)$ is effective for any sufficiently large $n$. 

\subsection{Endomorphism algebras of $F$-spaces}

Recall that a cyclic algebra over a field $K$ of any characteristic of degree $r$ is a triple 
\[ (K_r/K, \tau, a):=\bigoplus_{i=0}^{r-1} K_r \eta^i=K_r\<\eta\>, \] 
where 
$K_r/K$ is a cyclic extension of degree $r$ with Galois group $\Gal(K_r/K)=\< \tau \>$, $a\in K^\times$, and $\eta$ is an element satisfying the relations 
\[ \eta^r=a\quad  \text{and} \quad \eta x =\tau(x) \eta, \quad    x\in K_r. \] 
Every cyclic algebra $(K_r/K, \tau, a)$ is a central simple algebra over $K$ of degree $r$ (i.e.~of dimension $r^2$ over $K$). One has $(K_r/K, \tau, 1)\simeq \Mat_r(K)$, and 
\[ (K_r/K, \tau, a_1)\simeq (K_r/K, \tau, a_2) \iff a_1 a_2^{-1} \in \Nr_{K_r/K}(K^\times_r). \]  

If $K$ is a nonarchimedean local field, then every central simple algebra $A/K$ of degree $r$ is isomorphic to $(K_r/K,\sigma, a)$, where $K_r/K$ is the unramified extension of degree $r$, $\sigma$ the arithmetic Frobenius in $\Gal(K_r/K)$ and $a\in K^\times$. The invariant of $A$ is defined by
\[ \inv(A):=\frac{v(a)}{r} \ \ { \rm mod } \ \Z, \]
where $v$ is the normalized valuation of $K$ so that $v(\pi)=1$ for any uniformizer $\pi$ of $K$. 

Let $\lambda=s/r\in \Q$. Let $\Q_{p^r}$ be the unique unramified extension of $\Qp$ of degree $r$. There exist $a,b\in \Z$ such that $ar-bs=1$. Put
\[ \begin{split}
    K^\lambda &:=\Q_{p^r}\<\xi\>: 
     \ \xi^r=p, \ \xi \alpha=\sigma^{-b}(\alpha) \xi, \ \alpha\in \Q_{p^r}
    \\
    &=(\Q_{p^r/\Qp}, \sigma^{-b}, p).
\end{split} \]
Then $K^\lambda$ is a central simple algebra over $\Qp$ of degree $r$. Putting $\xi':=\xi^s$, one has
\[ (\xi')^r=p^s, \quad \xi' \alpha=\xi^s \alpha=\sigma^{-bs}(\alpha) \xi^s= \sigma(\alpha) \xi'. \]
and hence $K^\lambda=(\Q_{p^r}/\Qp, \sigma, p^s)$. Thus, 
\[ \inv(K^\lambda)=\lambda \ \ {\rm mod} \ \Z. \]

\begin{thm}\label{thm:endo}
    Let $\lambda=s/r$ and assume $k\supseteq \F_{p^r}$. Then $\End(N_k^\lambda)\simeq (K^\lambda)^{\opp}$ and it is a central division algebra over $\Qp$ of degree $r$ with invariant $-\lambda \mod \Z$.  
\end{thm}
\begin{proof} 
    Put $N:=L\otimes_{\Q_{p^r}} K^\lambda$, an $L$-vector space of dimension $r$ with basis $1\otimes \xi^i$, $i=0,\dots, r-1$, and define $F(b\otimes \xi^i):=\sigma(b)\otimes \xi^{i+s}.$ Then 
    \begin{itemize}
        \item[(i)] As $F$-spaces, one has 
        \[ (N,F)\simeq (L[\xi],F_s)\simeq N^\lambda_k.  \]
        \item[(ii)] $\End(N,F)$ is given by right multiplication by $K^\lambda$.
    \end{itemize}
    Statement (i) is clear and we show (ii). Let $f\in \End(N,F)$. Set 
    \[ N^\diamond:=\{x\in N: F^r(x)=p^s x\}=1\otimes K^\lambda,\]
    the subspace of constant sections of $p^{-s} F^r$. This is an $F$-space over $\F_{p^r}$. It is easy to check $f:N^\diamond\to N^\diamond$. As $N^\diamond=1\otimes K^\lambda$, $f$ is uniquely determined by its restriction $f|_{K^\lambda}$. On $K^\lambda$, $F$ acts as left multiplication by $\xi^s$ because
    \[ F(b \xi^i)=\sigma(b)\xi^{i+s}=\xi^s(b \xi^i), \quad b\in \Q_{p^r}. \]
    Therefore, $f$ commutes with left multiplication by $K^\lambda=\Q_{p^r}(\xi^s)$ and
    \[ \End(N^\diamond,F)=\End_{K^\lambda}(K^\lambda)=\{\text{right multiplication by } K^\lambda\}. \]
    This proves $\End(N,F)\simeq (K^\lambda)^\opp$. \qed
\end{proof}

\begin{remark}\label{rem:26} \ 

\begin{itemize}
    \item[(i)] $N^\lambda_k$ is effective if and only if $\lambda\ge 0$.
    \item[(ii)] $N^\lambda_k$ arises from a $p$-divisible group if and only if $0\le \lambda \le 1$.
    \item[(iii)] If $\lambda\neq \lambda'$, then $\Hom(N^\lambda_k, N^{\lambda'}_k)=0$.    
\end{itemize}
If $M\subset N^\lambda_k$ is an $F$-lattice and let $\{p^{a_1}, \dots, p^{a_r}\}$ be the elementary divisors of the $W$-sublattice $FM$ in $M$, then $a_i\ge 0$ and $\sum_i a_i=s$. This shows $\lambda=s/r \ge 0$. If $pM \subset FM\subset M$, then $0\le a_i\le 1$ and $0\le \lambda \le 1$. Conversely, if $0\le \lambda \le 1$, then $W[F,V]/(F^r-V^s)$ is a \dieu module in $N^\lambda_k$. This shows (i) and (ii).

To see (iii), let $f\in \Hom(N^\lambda_k,N^{\lambda'}_k)$. Since $N^\lambda_k=L\otimes N^\lambda$, it suffices to show $f(N^\lambda)=0$. Write $\lambda=s/r$ and $\lambda'=s'/r'$. For any $x\in N^\lambda$, put $y=f(x)$. As $F^{rr'} x=p^{sr'}x$, one has $F^{rr'}y=p^{sr'} y$. 
Write $y=\sum_i a_i \xi^i$ with $a_i\in L$. It follows from $F^{rr'}y=p^{sr'} y$ and $F^{r'} \xi^i=p^{s'} \xi^i$ that
\[ p^{sr'} a_i=\sigma^{rr'}(a_i) p^{rs'} \] 
for all $i$. As $sr'\neq rs'$, one has $a_i=0$ for all $i$, and $y=0$.   
\end{remark} 

\subsection{The Manin-\dieu Theorem}\label{sec:4.3}

\begin{thm} \label{thm:manin} Assume that $k$ is algebraically closed.  

{\rm (1)} The category of $F$-spaces over $k$ is semi-simple. 

{\rm (2)} Every simple object is isomorphic to $N^\lambda_k$ for some $\lambda\in\Q$.

{\rm (3)} Let $N$ be an $F$-space over $k$. Then there are rational numbers $\lambda_i$ and positive integers $m_i$ for $i=1,\dots h$, and there is an isomorphism 
\[ N\simeq \bigoplus_{i=1}^h (N^{\lambda_i}_k)^{m_i} \]
of $F$-spaces. Moreover, the set $\{(\lambda_i,m_i)\}_{1\le i\le h}$
is uniquely determined by $N$.   
\end{thm}

The proof consists of the following five lemmas.

\begin{lemma}\label{lm:4.0}
  For any integers $\alpha, \beta$, the map
  \[ \varphi: L \to L, \ x\mapsto p^\beta x^{(\alpha)} -x \]
  is surjective, where we write $x^{(n)}:=\sigma^n(x)$.
\end{lemma}
\begin{proof}
   We need to solve the equation ($*$) $p^\beta x^{(\alpha)} -x =b$ for each $b\in L$. 

   Suppose $\beta>0$. Take
   \[ x=-(b+ p^\beta b^{(\alpha)} + p^{2\beta} b^{(2\alpha)}+\cdots) .\]
   We find $x-p^\beta x^{(\alpha)}=-b$ and thus equation $(*)$ is solvable.

   Suppose that $\beta<0$. Put $x':=p^\beta x^{(\alpha)}$. Then $x=p^{-\beta} x'^{(-\alpha)}$ and 
   \[ x'- p^{-\beta} x'^{(-\alpha)}=p^\beta x^{(\alpha)}-x=b.\]
   As $-\beta>0$, this equation is solvable for $x'$ and hence so is equation $(*)$.

   Suppose that $\beta=0$. We solve the equation by successive approximation. We may assume $b\in W$. Suppose $x^{(\alpha)}-x-b =p^m a_m$ for some $a_m\in W$. Put $x'=x+p^m y$ and we have 
   \[ x'^{(\alpha)}-x'-b =p^m (a_m+y^{(\alpha)}-y). \]
   
   Since $k$ is algebraically closed, there exists $y\in W$ such that \[ a_m+y^{(\alpha)}-y\!\!\! \mod {pW} = \bar a_m+\bar y^{p^\alpha}-\bar y=0, \] 
   where $\bar y$ denotes the image of $y$ in $k$.
   Replacing $x$ by $x'$, one has $x^{(\alpha)}-x-b =p^{m+1} a_{m+1}$ for some $a_{m+1}\in W$. As $W$ is complete, the equation $x^{(\alpha)}-x-b=0$ is solvable. \qed 
\end{proof}

\begin{lemma}\label{lm:4.1}
   Any short exact sequence of $F$-spaces of the form
   \[
   \begin{tikzcd}
       0 \arrow[r] & N' \arrow[r] & N \arrow[r,"\varphi"] & N_k^{\lambda} \arrow[r] & 0
   \end{tikzcd} \]
   where $\lambda\in \Q$ and $N'=\bigoplus_{\lambda'} (N_k^{\lambda'})^{m_{\lambda'}}$ is a direct sum of $N_k^{\lambda'}$, splits. 
\end{lemma}
\begin{proof}
    Replacing the exact sequence by a suitable twisting, we may assume that $\lambda=s/r$ and all $\lambda'=s'/r'$ are positive. 

    We first show that the map 
    \begin{equation}\label{eq:4.1}
        F^r-p^s: N_k^{\lambda'} \to N_k^{\lambda'}
    \end{equation}
    is surjective. Since $F$ commutes with $p$, one has
    \[ F^{rr'}-p^{sr'}=(F^r-p^s)(F^{r(r'-1)}+F^{r(r'-2)}p^{s}+\dots+p^{s(r'-1)}). \]
    If the map $F^{rr'}-p^{sr'}$ is surjective, then so is the map $F^r-p^s$. Thus, it suffices to show that the map $F^{rr'}-p^{sr'}:N_k^{\lambda'} \to N_k^{\lambda'}$ is surjective. Let $e_1',\dots , e_{r'}'$ be an $L$-basis for $N_k^{\lambda'}$, and one has $F^{rr'}e_i'=p^{rs'}e_i'$. One computes
    \[ (F^{rr'}-p^{sr'})\left(\sum_i a_ie'_i\right )=\sum_i (a_i^{(rr')}p^{rs'}-p^{sr'}a_i) e_i',\]
    where $a_i\in L$. Set $x=p^{sr'}a_i$, then we have 
    \[ p^{rs'-sr'} x^{(rr')}-x=a_i^{(rr')}p^{rs'}-p^{sr'}a_i.\]
    Since the equation $p^{rs'-sr'} x^{(rr')}-x=b_i$ is solvable for any $b_i\in L$ by Lemma~\ref{lm:4.0},  there exists $a_i$ such that $a_i^{(rr')}p^{rs'}-p^{sr'}a_i=b_i$. Thus, the map $F^r-p^s$ is surjective. 

    To show the short exact sequence splits, we need to construct an element $x\in N$ such that $(F^r-p^s)(x)=0$. Then the map $N_k^\lambda \to N$, $e_1 \mapsto x$ gives a section of $\varphi$. 
    
    Choose an element $x\in N$ such that $\varphi(x)=e_1$. Since $(F^r-p^s)(e_1)=0$, one has $\varphi((F^r-p^s)(x))=0$ and $(F^r-p^s)(x)\in N'$. As the map $F^r-p^s$ in \eqref{eq:4.1} is surjective, there exists an element $x'\in N'$ such that $(F^r-p^s)(x')=(F^r-p^s)(x)$.  Replacing $x$ by $x-x'$, one has $(F^r-p^s)(x)=0$. This completes the proof of the lemma. \qed    
\end{proof}

Let $R:=L[F]$ denote the non-commutative polynomial ring over $L$ with relation $F a=\sigma(a) F$ for $a\in L$, and $W[F]$ be its subring with coefficients in $W$. Let $W[p^{1/r}][F]$, for $r\ge 1$, denote the non-commutative polynomial ring over $W[p^{1/r}]$ with relations $F a=\sigma(a) F$ for $a\in W$ and $F p^{1/r}=p^{1/r} F$.

\begin{lemma} \label{lm:4.2}
   Let $F^n+a_1 F^{n-1}+\dots +a_n\in W[F]$ with $n\ge 1$. Then there exist co-prime integers $r\ge 1$ and $s$, and elements $b_0,\dots ,b_{n-1}, u\in W[p^{1/r}]$ with $u$ invertible such that 
   \begin{equation}\label{eq:4.2}
       F^n+a_1 F^{n-1}+\dots +a_n=(b_0 F^{n-1}+\dots +b_{n-1})(F-p^{s/r}) u 
   \end{equation}
   in $W[p^{1/r}][F]$.
\end{lemma}
\begin{proof}
    Let $\lambda:=\min \{\frac{v(a_1)}{1},\dots \frac{v(a_n)}{n} \}$. Then $i\cdot \lambda\le v(a_i)$ for all integers $1\le i \le n$, and $i\cdot \lambda=v(a_i)$ for some $i$. Write $\lambda=s/r$ and $a_i= p^{is/r} \alpha_i$ for each $i$, then $\alpha_i\in W[p^{1/r}]$ and $\alpha_i$ is a unit for some $i$. Put $b_i=p^{is/r} \beta_i$, $1\le i \le n$, and $v=u^{-1}$. Multiplying $v$ by the left hand side of \eqref{eq:4.2}, one has
    \[ (F^n+a_1 F^{n-1}+\dots +a_n)v=v^{(n)} F^n+v^{(n-1)} a_1 F^{n-1} +\dots+v a_n. \]
    The corresponding right hand side is  
    \[ b_0 F^n+(b_1-p^{s/r} b_0)F^{n-1}+(b_2-p^{s/r} b_1) F^{n-2}+ \dots + (b_{n-1}-p^{s/r} b_{n-2})F-b_{n-1} p^{s/r}. \]
    Comparing the coefficients, one obtains
    \[ v^{(n)}=b_0, \quad v^{(n-1)} a_1=b_1-p^{s/r} b_0,\dots, \]
    \[ v^{(1)} a_{n-1}=b_{n-1}-p^{s/r} b_{n-2}, \quad va_n=-p^{s/r} b_{n-1},\]
    or equivalently,
    \begin{equation} \label{eq:4.21}   
     \begin{split}
        v^{(n)} &=\beta_0,\\
        \alpha_1 v^{(n-1)}& =\beta_1-\beta_0,\\
           & \vdots \\
           \alpha_{n-1} v^{(1)}&=\beta_{n-1}-\beta_{n-2}\\
           \alpha_n v&=-\beta_{n-1}.
     \end{split}
     \end{equation}
Thus, Equations~\eqref{eq:4.21} are solvable if and only if there exists a unit $v\in W[p^{1/r}]$ such that 
\begin{equation}\label{eq:4.22}
   P(v):=v^{(n)}+ \alpha_1 v^{(n-1)} +\dots+\alpha_{n-1}v^{(1)}+\alpha_n v=0. 
\end{equation}
We solve Equation~\eqref{eq:4.22} by successive approximation. Clearly, the equation 
$P(\bar v)=\bar v^{p^n}+\bar \alpha_1 \bar v^{p^{n-1}}+\dots +\bar \alpha_n \bar v=0$
has a nonzero solution $\bar v$ in $k$ as $\bar \alpha_i\neq 0$ for some $i$. Thus, one has $P(v_1)\equiv 0 \mod p^{1/r}$, where $v_1$ is a lift of $\bar v$.

Suppose we have $P(v_k)=z p^{k/r}$ for some unit $v_k\in W[p^{1/r}]$.
Put $v_{k+1}=v_k+p^{k/r}x$, where $x\in W[p^{1/r}]$. We need to solve
\[ P(v_{k+1})=p^{k/r}(P(x)+z) \equiv 0 \mod p^{(k+1)/r}, \]
or equivalently to solve the equation 
\[ \bar x^{p^n}+\bar \alpha_1 \bar x^{p^{n-1}}+\dots +\bar \alpha_n \bar x +\bar z=0\]
in $k$, which is solvable. Since $W[p^{1/r}]$ is complete, the equation $P(v)=0$ has a unit solution. \qed
\end{proof}

\begin{lemma}\label{lm:4.3}
   Let $N$ be a nonzero $F$-space. Then there exist a number $\lambda\in \Q$ and a nonzero map $N\to N^\lambda_k$. 
\end{lemma}
\begin{proof}
    Replacing $N$ by its simple quotient, we may assume that $N$ is a simple $R$-module. Since $R=L[F]$ is a non-commutative euclidean ring, $N\simeq R/RP$ for a monic polynomial $P\in R$. Replacing $N$ by a suitable $N(-m)$, we may write $P=F^{n}+a_1 F^{n-1}+\dots +a_n$ with $a_i\in W$ for all $i$. By Lemma~\ref{lm:4.2}, we have 
    \[ P=Q (F-p^{s/r}) u \quad \text{in}\ W[p^{1/r}][F] \]
    for co-prime integers $r\ge 1$, $s\ge 0$, a polynomial $Q\in W[p^{1/r}][F]$ and an invertible element $u\in W[p^{1/r}]^\times$. 
    Put 
    \[ R_r:=L[p^{1/r}][F]. \] 
    Observe that $R_r/R_r (F-p^{s/r})$ is an $L$-vector space with basis $1,p^{1/r}, \dots, p^{(r-1)/r}$ and $F(p^{i/r})=p^{(i+s)/r}$. Therefore, $R_r/R_r(F-p^{s/r})\simeq N^{s/r}_k$. Consider the maps
    \[ \begin{split}
        N=R/R\, P & \embed L[p^{1/r}]\otimes_{L} R/R\, P=R_r/R_r \, Q(F-p^{s/r})u  \\
        & \isoto R_r/R_r Q(F-p^{s/r}) \onto R_r/R_r (F-p^{s/r})=N^{s/r}_k,
    \end{split}
    \]
    where the isomorphism in the second row is the map $x\mapsto xu^{-1}$ and the last map is the natural projection. Then the composition gives a nonzero map $N \to N^{s/r}_k$ as $1 \mapsto u^{-1}\neq 0$. \qed   
\end{proof}

\begin{lemma}\label{lm:4.4}
   Every $F$-space $N^\lambda_k$ is simple. Conversely, every simple $F$-space $N$ is isomorphic to $N^\lambda_k$ for some $\lambda\in \Q$.
\end{lemma}
\begin{proof}
    Suppose that $N\subsetneq N^\lambda_k$ is a proper $F$-subspace. 
    As $N^\lambda_k/N\neq 0$, by Lemma~\ref{lm:4.3}, there exist an element $\mu\in \Q$ and a commutative diagram of $F$-spaces
    \[ \begin{tikzcd}
        N^\lambda_k \arrow[d] \arrow[dr, "f"] & \\
        N^\lambda_k/N \arrow[r, "g"] & N^\mu_k
    \end{tikzcd}\]
    in which $g$ is a nonzero map. 
    Since $f$ is nonzero, one has $\mu=\lambda$ by Remark~\ref{rem:26} (iii), and by Theorem~\ref{thm:endo}, $f$ is an isomorphism. Thus $N=0$ and $N^\lambda_k$ is simple. 

    By Lemma~\ref{lm:4.3}, there exists a nonzero map $N\to N^\lambda_k$. Since both $N^\lambda_k$ and $N$ are simple, this map is an isomorphism. 
    \qed
\end{proof}

\begin{remark}
    The first part of Lemma~\ref{lm:4.4} holds true if $k$ is replaced by any perfect field containing $\F_{p^r}$. Indeed, if $0\neq N\subsetneq N^\lambda_k$ is a proper $F$-subspace over $k$, then 
    $0\neq N_{\bar k} \subsetneq N^\lambda_{\bar k}$ is a proper $F$-subspace over $\ol k$, a contradiction as $N^\lambda_{\bar k}$ is simple.
\end{remark}

\npr {\bf Proof of Theorem~\ref{thm:manin}.}
Statement (2) follows from Lemma~\ref{lm:4.4}. We prove statement (3) by induction on the length of $N$. If $N$ is simple, then there is nothing to show and we may assume $N$ is not simple. 
By Lemmas~\ref{lm:4.3} and \ref{lm:4.4}, we have a short exact sequence of $F$-spaces
\[ 0 \longrightarrow N'\longrightarrow N \longrightarrow N_k^\lambda \longrightarrow 0\]
for some $\lambda\in \Q$. By the induction hypothesis, $N'$ is a direct sum of $N_k^{\lambda'}$ and hence by Lemma~\ref{lm:4.1} this short exact sequence splits. 
It follows that 
\begin{equation}
    \label{eq:dec} 
     N\simeq \bigoplus_{\lambda\in \Q} (N^\lambda_k)^{m_\lambda},
\end{equation}
where the multiplicity $m_\lambda=0$ for almost all $\lambda$. By Theorem~\ref{thm:endo}, for each $\lambda\in \Q$, one has \[ \Hom(N^\lambda_k, N)=((K^\lambda)^\opp)^{m_\lambda}. \]  
Therefore, the multiplicity $m_\lambda$ is uniquely determined by $N$ and the decomposition \eqref{eq:dec} is unique up to permutation. \qed

\section{Supersingular abelian varieties}

Our first aim is to prove the equivalence of the following
definitions for supersingular abelian varieties \cite[Theorem 4.2]{oort:subvar}:

\begin{defn}
  An abelian variety $X$ over a field $K$ is called
  \emph{supersingular} if it satisfies one of the following equivalent
  properties
  \begin{itemize}
  \item [(a)] $X$ has single slope $1/2$.
  \item [(b)] $X$ is isogenous to a product of supersingular elliptic
    curves over the  algebraic closure $\ol K$.     
  \end{itemize}
\end{defn}

The direction condition (b) $\implies$ 
condition (a) is obvious. The other direction will be proved in Proposition~\ref{prop:14}.
For the moment, let us call $X$ \emph{supersingular} if it satisfies condition (a).

We shall also give a different proof of the following well-known results due to Oort \cite[Theorem 2]{oort:product}, and to Deligne, Ogus and Shioda (cf.~{\cite[Theorem 6.2]{ogus:ss}} and {\cite[Theorem 3.5]{shioda}}). 

In this section, $k$ denotes an \ac field.
\begin{thm}[Oort]
  \label{thm:oort}
  Let $X$ be a $g$-dimensional abelian variety over $k$ 
  and suppose that $a(X)=g$. Then there exist supersingular
  elliptic
  curves $E_1,\dots, E_g$ and an isomorphism $X\simeq E_1\times
  \dots \times E_g$.  
\end{thm}

\begin{thm}[Deligne, Ogus, Shioda] \label{thm:DOS}
  For any integer $g>1$ and  any $2g$ 
  supersingular elliptic
  curves $E_1,\dots, E_{2g}$ over $k$, one has
  \begin{equation}
    \label{eq:4}
    E_1\times \dots \times E_g\simeq E_{g+1}\times \dots \times E_{2g}.
  \end{equation}
\end{thm}
For an abelian scheme $X\to S$ of relative dimension $g$ over a base scheme $S$ of characteristic $p$, we define the \emph{$a$-number $g$ locus} $S(a=g)$ of $S$ to be the maximal closed subscheme $T\subseteq S$ such that the restriction map $V_{X_T/T}: \Lie(X^{(p)}_T/T) \to \Lie(X_T/T)$ to $T$ vanishes. 

Let $g,d,n$ be positive integers such that $n\ge 3$ and $(n,pd)=1$, and let $\calA_{g,d,n}$ be the moduli space over $\Fpbar$ of $g$-dimensional polarized abelian varieties $(X,\lambda,\eta)$ of polarization degree $d^2$ with level-$n$ structure.   For the definition of a level-$n$ structure $\eta$ on a polarized abelian variety $(X,\lambda)$, we refer to \cite[Section 2.1, p.~453]{yu:ss_siegel}. $\calA_{g,d,n}$ is a fine moduli scheme over $\Fpbar$ and it admits a universal family $(\calX,\lambda_\calX, \eta_\calX) \to \calA_{g,d,n}$ of polarized abelian varieties with level-$n$ structure. Namely, for any polarized abelian scheme $(X,\lambda,\eta)$ of relative dimension $g$ with polarization degree $d^2$ with level-$n$ structure over an $\Fpbar$-scheme $S$, there exist a unique morphism $f:S\to \calA_{g,d,n}$ of $\Fpbar$-schemes and a unique isomorphism $(X,\lambda,\eta)\isoto f^*(\calX,\lambda_\calX, \eta_\calX)$ of such geometric objects over $S$.   
Let $\calA_{g,d,n}(a=g)$ be the $a$-number $g$ locus for the family $\calX\to \calA_{g,d,n}$ of abelian varieties.
Note that a member $(X,\lambda,\eta)\in \calA_{g,d,n}(k)$ has $a$-number $g$ if and only if the moduli map $\Spec k \to \calA_{g,d,n}$ factors through the closed subscheme $\calA_{g,d,n}(a=g)$.


\begin{lemma}\label{lm:12}
    The subscheme $\calA_{g,d,n}(a=g)$ is finite \'etale over $\Fpbar$.
\end{lemma}
\begin{proof}
    Let $(X_0,\lambda_0,\eta_0)\in \calA_{g,d,n}(a=g)(\Fpbar)$ be a member and let ${\rm Def}(X_0,\lambda_0)\subseteq {\rm Def}(X_0)$ denote the respective deformation functors of $(X_0,\lambda_0)$ and $X_0$. It suffices to show that the subfunctor ${\rm Def}(X_0)(a=g)=\Spec \Fpbar$.
    A theorem of Serre-Tate states that deforming abelian varieties amounts to deforming their $p$-divisible groups. Thus, it is equivalent to show ${\rm Def}(M_0)(a=g)=\Spec \Fpbar$ by Theorem~\ref{thm:cartier},
    where $M_0=\bfD_*(X_0)$ is the covariant \dieu module of $X_0$. 

    As $a(X_0)=g$, we can choose a $W$-basis $e_1,\dots e_{2g}$ of $M_0$ such that 
    \[ F e_j=e_{g+j}, \quad e_{g+j}=V e_j, \quad \forall\, 1\le j\le g. \]
    By Theorem~\ref{thm:defo}, the universal \dieu module $M_R$ over the universal first order deformation ring $R=\Fpbar [t_{ij}]_{1\le i,j\le g}/(t_{ij})_{1\le i, j\le g}^2$ can be constructed by the following generators and relations:
    \[  F e_j=e_{g+j}+\sum_{i=1}^g T_{ij} e_i, \quad e_{g+j}=V e_j, \quad \forall\, 1\le j\le g, \]
    where $T_{ij}=(t_{ij},0,\dots)\in W(R)$ is the Teichm\"uller lifting of $t_{ij}$. Thus, $F(\bar e_j)=\sum_i t_{ij} \bar e_i$ on $M_R/VM_R$, where $\bar e_i$ and $\bar e_j$ are the image of $e_i$ and $e_j$ in $M_R/VM_R$, and the vanishing locus of the map $F: M_R/V M_R \to M_R/ V M_R$ (which corresponds to the vanishing locus of the Verschiebung map $V$ on the Lie algebra of the universal deformation) is defined by $t_{ij}=0$ for all $i,j$. This shows ${\rm Def}(X_0)(a=~\!g)=\Spec \Fpbar$ and the lemma. \qed
\end{proof}


\begin{lemma}\label{lm:13}
  Let $X$ be as in Theorem~\ref{thm:oort}. Then there is an abelian variety
  $X_0$ defined over a finite field $\Fq$ and there is an isomorphism
  $X\simeq X_0 \otimes k$. Moreover, one can choose $X_0/\Fq$ so that the Frobenius endomorphism $\pi_{X_0}$ is equal to $\sqrt{q}\in \Q$, and hence that $X_0\otimes \Fpbar$ is isogenous to $E_0^g  \otimes \Fpbar$, where $E_0$ is a supersingular elliptic curve over $\F_{p}$ with $\pi_{E_0}^2=-p$. 
\end{lemma}
\begin{proof}
  Choose a polarization $\lambda$ and a level-$n$ structure $\eta$ on $X$, and make a member $(X,\lambda,\eta)\in \calA_{g,d,n}(k)$ for suitable integers $g,d,n$. Since $a(X)=g$, one has 
  $(X,\lambda,\eta)\in \calA_{g,d,n}(a=g)(k)=\calA_{g,d,n}(a=g)(\Fpbar)$ by  Lemma~\ref{lm:12}. Thus, there is a member $(X_0,\lambda_0,\eta_0)$ over a finite field $\Fq$ and there is an isomorphism  $(X,\lambda,\eta)\simeq (X_0,\lambda_0,\eta_0)\otimes k$. Since $X_0$ has single slope $1/2$, the Frobenius endomorphism has
  all eigenvalues $\sqrt{q} \zeta$, for some root $\zeta$ of
  unity \cite[p.~116]{oort:subvar}. Replacing $X_0/\Fq$ by a suitable base change $X_0\otimes \F_{q^m}$, we may assume that $\pi_{X_0}=(\pi_{E_0^g})^{[\Fq:\Fp]}= \sqrt{q}\in \Q$. Then $X_0$ and $E_0^g \otimes \Fq$ are isogenous over $\Fq$ by a theorem of Tate~\cite[Theorem 1]{tate:eav} as their Frobenius endomorphisms have the same characteristic polynomial.\qed   
\end{proof}

\begin{prop}\label{prop:14}
  Any $g$-dimensional abelian variety $X$ over $k$
  with single slope $1/2$ is isogenous to $E_0^g\otimes_{\Fp} k$, where $E_0$ is as in Lemma~\ref{lm:13}. In particular any two such abelian varieties over $k$ are isogenous.
\end{prop}
\begin{proof}
  By the Manin-\dieu Theorem, $X$ is isogenous to an abelian variety
  $Y$ over $k$ satisfying $a(Y)=g$. As $Y$ is isogenous to $E_0^g\otimes k$ by Lemma~\ref{lm:13}, $X$ is isogenous to $E_0^g\otimes k$.  \qed   
\end{proof}

\begin{lemma}\label{lm:15}
  Let $X$ be a 
  supersingular abelian variety over $k$ of dimension $g>1$, and $O:=\End(X)$ the endomorphism ring of $X$. Then there are natural 
  bijections between the following three finite sets:
  \begin{itemize}
  \item[(a)]  The set $\Lambda_X$ of isomorphism classes of supersingular abelian 
  varieties $X$ over $k$ satisfying $X'[p^\infty]\simeq X[p^\infty]$.

  \item[(b)] The set $\Cl(O)$ of isomorphism classes of locally free right ideals of $O$. 
  
  \item[(c)] The set $\Zp^\times/\Nr(O_p^\times)$, where $O_p=O\otimes_{\Z} \Zp$ and $\Nr: O_p^\times \to \Zp^\times$ denotes the reduced norm map.  
  \end{itemize}
\end{lemma}
\begin{proof}
  That the sets (a) and (b) are bijective
  is the classical formulation of the special case of 
  \cite[Theorem 2.1 and Proposition 2.2]{yu:mass_hb} without polarization; also see its generalizations in \cite{xue-yu:counting,yu:smf}. We provide the proof for the reader's convenience. We define the automorphism group scheme $G$ of $X$ over $\Z$ by
  \[ G(R):=(O\otimes R)^\times\]
  for any commutative ring $R$. We have an injective map 
  \begin{equation}\label{eq:3.2}
     \End(X)\otimes \Z_\ell\to \End(T_\ell(X)) 
  \end{equation}
  of finite free $\Z_\ell$-modules for any prime $\ell\neq p$, where $T_\ell(X)$ is the Tate module of $X$, and an injective map 
  \begin{equation}\label{eq:3.3}
     \End(X)\otimes \Z_p\to \End_{\rm DM}(\bfD_*(X)) 
  \end{equation}
  of finite free $\Z_p$-modules, where $\bfD_*(X)$ is the covariant \dieu module of $X$.  See \cite[Section 19, Theorem 3, p.~176]{mumford:av} for the statement with $\ell\neq p$; the same argument together with Theorem~\ref{thm:dieu} works well for $\ell=p$.  By Proposition~\ref{prop:14}, $\End^0(X)\simeq \Mat_g(\End^0(E_0\otimes k)), $ and hence it has dimension $4g^2$ over $\Q$. The dimension count implies that the injective maps in \eqref{eq:3.2} and \eqref{eq:3.3} are isomorphisms. Thus, we may identify $G(\Q_\ell)=\Aut(V_\ell(X))$ and $G(\Q_p)=\Aut_{\rm DM}(N)$, where $V_\ell(X):=T_\ell(X)\otimes \Q_\ell$ and $N:=\bfD_*(X)\otimes \Qp$. 

  Let $\wt \Lambda_X$ be the set of pairs $(X',\varphi)$, where $X'$ is an abelian variety over $k$ and $\varphi:X'\to X$ is a quasi-isogeny, that is, $\varphi$ is an element in $\Hom(X',X)\otimes \Q$ and $m \varphi$ is an isogeny for some integer $m$. Using Proposition~\ref{prop:14}, one easily checks that the map $(X',\varphi)\mapsto [X']$, the isomorphism class of $X'$ over $k$, induces a bijection 
  $G(\Q)\backslash \wt \Lambda_X \isoto \Lambda_X$. On the other hand, to each $(X',\varphi)$ we attach a $\Z_\ell$-lattice $L_\ell$ for primes $\ell$ including $p$ by 
  \[ L_\ell:=\varphi_*(T_\ell(X'))\subseteq V_\ell(X)\ (\ell \neq p), \quad L_p:=\varphi_*(\bfD_*(X'))\subseteq N, \]
  and we have $L_\ell=T_\ell(X)$ for almost all $\ell$. This shows that $\wt \Lambda_X$ is in bijection with the set of collections of $\Z_\ell$-lattices satisfying the above condition. For each $L_\ell$ there exists a unique coset $g_\ell G(\Z_\ell) \in G(\Q_\ell)/G(\Z_\ell)$ such that $L_\ell=g_\ell T_\ell(X)$ or  $L_p=g_p D_*(X)$. This yields a natural bijection $\wt \Lambda_X\isoto G(\A_f)/G(\wh \Z)$, $(X',\varphi)\mapsto [(g_\ell)_{\ell}]$ and natural bijections 
  \begin{equation}\label{eq:3.4}
      \Lambda_X\isoto G(\Q)\backslash G(\A_f)/G(\wh \Z)\simeq \Cl(O).  
  \end{equation}
  We remark that this step also works for $g=1$.    
  
  We now show a bijection of $\Cl(O)$ with $\Zp^\times/\Nr(O_p^\times)$. The reduced norm map induces a short exact sequence of algebraic groups 
\[\begin{tikzcd}
1 \arrow[r] & G^{\der}_\Q \arrow[r] & G_\Q \arrow[r] & \Gm \arrow[r] & 1.
\end{tikzcd}
\]  
   The derived group $G^{\der}_\Q$ is semi-simple and simply connected, and $G^{\der}(\R)$ is non-compact as $g>1$. 
   Kneser's Theorem $H^1(\Q_\ell, G^{\der}_\Q)=1$ \cite[Theorem 6.4, p.~284]{platonov-rapinchuk:agnt} implies that $\Nr(G(\Q_\ell))=\Q_\ell^\times$ and hence that $\Nr(G(\A_f))=\A_f^\times$. By the Hasse Norm Theorem (see \cite[(32.9) Theorem, p.~275 and (32.20) Theorem, p.~280]{reiner:mo}, we have $\Nr(G(\Q))=\Q^\times_{>0}$. 
   Then strong approximation implies 
  \[ \Nr: G(\Q)\backslash G(\A_f)/G(\wh \Z) \isoto \Q^\times_{>0} \backslash \A_f^\times /\Nr(\wh O^\times)=\Zp^\times/\Nr(O_p^\times); \]
  see \cite[Section 4]{yu:swan} for more details. 
  This proves the lemma. \qed
\end{proof}


\begin{remark}\label{rem:corr}
    (1) Note that the proof of bijectivity of \eqref{eq:3.2} and \eqref{eq:3.3} is elementary; it does not require Tate's theorem \cite[Main Theorem]{tate:eav}. Thus, the proof of the bijection (a)$\simeq$(b) is really elementary; it only requires Proposition~\ref{prop:14}. 

    (2) The statement \eqref{eq:3.4} is the generalization of \eqref{eq:1.1}, and the description (c) is a generalization of \eqref{eq:1.2}.
\end{remark}

\begin{prop}\label{prop:17}
    If $X=E\times X_1 $ is the product of an elliptic curve $E$ and a supersingular abelian variety $X_1$ of positive dimension, then $\Lambda_X$ is a singleton. 
\end{prop}
\begin{proof}
    Since $O\supseteq \End(E)\times \End(X_1)$, we have $\Nr(O_p^\times)\supseteq \Nr((\End(E)\otimes \Zp)^\times\times \{1\})=\Zp^\times$ and hence $\Lambda_X$ is a singleton by Lemma~\ref{lm:15}.  \qed
\end{proof}

\npr {\bf Proof of Theorem~\ref{thm:DOS}.} Take $X=E_1\times (E_2\times \dots \times E_g)$ and $X'=E_{g+1}\times \dots \times E_{2g}$. Since
$X'[p^\infty]\simeq X[p^\infty]$, we have $X'\in \Lambda_X$ and get $X'\simeq X$ by Proposition~\ref{prop:17}. \qed

\npr {\bf Proof of Theorem~\ref{thm:oort}.}
There is nothing to prove if $g=1$ so we assume $g>1$.
Take $X'=E_1\times \dots \times E_g$, where $E_i$, $1\le i\le g$, are any supersingular elliptic curves. As $a(X)=a(X')=g$, we have $X'[p^\infty]\simeq X[p^\infty]$. This shows $X\in \Lambda_{X'}$ and hence $X'\simeq X$ by Proposition~\ref{prop:17}. \qed

With Theorems~\ref{thm:oort} and~\ref{thm:DOS}, we make the following definition:
\begin{defn}
  An abelian variety $X$ over a field $K$ is called
  \emph{superspecial} if it satisfies one of the following equivalent
  properties
  \begin{itemize}
  \item [(a)] $a(X)=g$;
  \item [(b)] $X\otimes \ol K$ is isomorphic to $E^g$ for a supersingular elliptic curve $E$ over $\ol K$.  
  \end{itemize}
\end{defn}

We consider a supersingular abelian surface $X$ over $k$ with $a(X)=1$. 
Then there is a minimal isogeny $\varphi:E_0^2\otimes k \to X$ of degree $p$, where $E_0/{\F_{p}}$ is a supersingular elliptic curve with $\pi_{E_0}^2=-p$; see \cite[Lemma 1.8]{li-oort}.  This isogeny gives rise to an embedding $\iota: \alpha_p \embed \alpha_p^2 \subseteq E_0^2 \otimes k$ with $\iota(\alpha_p)=\ker \varphi$ and each embedding $\iota: \alpha_p \embed \alpha_p^2$ corresponds to a parameter $t\in \bbP^1(k)$. The quotient abelian surface $E_0^2\otimes k/\iota(\alpha_p)$ is superspecial if and only if $t\in \bbP^1(\F_{p^2})$. Thus for $a(X)=1$, we have $t\in \bbP^1(k)\setminus \bbP^1(\F_{p^2})$. We call $t\notin \F_{p^4}$ the first case (I) and $t\in \F_{p^4}\setminus \F_{p^2}$ the second case (II). Such a distinction does not depend on the choice of the representation $(E_0^2 \otimes k,\varphi)$. 

\begin{prop}\label{prop:19}
   (1) Let $X$ be a supersingular abelian surface over $k$ with $a(X)=1$. Then 
   \begin{equation}\label{eq:19}
      \Lambda_X\simeq \begin{cases}
       \F_p^\times/(\F_{p}^\times)^2, & \text{Case (I)}; \\
       \{1\}, &  \text{Case (II)}. 
   \end{cases} 
   \end{equation}
   (2) Any two non-superspecial supersingular abelian surfaces over $k$ in Case (II) are isomorphic.
\end{prop}
\begin{proof}
    (1) The minimal isogeny $\varphi: E_0^2\otimes k\to X$ gives rise to an inclusion $\End(X) \embed \End(E_0^2\otimes k)$. Write $O_p:=\End(X)\otimes \Zp$ and $\End(E_0\otimes k)\otimes \Zp=\Z_{p^2}[\Pi]$ with relations $\Pi^2=-p$ and $\Pi a=\bar a \Pi$ for $a\in \Z_{p^2}$. The reduction modulo $\Pi$ gives a surjective map $m_\Pi:\Mat_2(\Z_{p^2}[\Pi])\to \Mat_2(\F_{p^2})$ and we have the commutative diagram:
    \[ \begin{tikzcd}
        \Mat_2(\Z_{p^2}[\Pi]) \arrow[r, "m_\Pi"] \arrow[d, "\Nr"] & \Mat_2(\F_{p^2}) \arrow[d, "\Nr_{\F_{p^2}/\Fp} \circ \det"] \\
        \Zp \arrow[r, "m_p"] & \Fp. 
    \end{tikzcd}
    \]
    One can check directly that the reduced norm map on $A\in \Mat_2(\Z_{p^2})\subseteq \Mat_2(\Z_{p^2}[\Pi])$ is given by $\Nr(A)=\det A \cdot \sigma(\det A)$ and the commutative diagram above follows.
    We also have a 
    short exact sequence 
    \[ 1 \longrightarrow 1+\Pi \Mat_2(\Z_{p^2}[\Pi]) \longrightarrow O_p^\times \longrightarrow m_{\Pi}(O_p^\times) \longrightarrow 1.\]
    We have $\Nr(1+\Pi \Mat_2(\Z_{p^2}[\Pi]))=1+p\Zp$, as it contains $\Nr \begin{bmatrix}        
    \begin{smallmatrix}
        1+\Pi \Z_{p^2}[\Pi] & 0 \\
        0 & 1
    \end{smallmatrix}\end{bmatrix}= \Nr(1+\Pi \Z_{p^2}[\Pi])=1+p\Zp$. It then follows that 
    \[ \Zp^\times /\Nr(O_p^\times)=\Fp^\times/\Nr_{\F_{p^2}/\Fp} \circ \det(m_\Pi(O_p^\times)). \] 
    
    According to \cite[Proposition 3.2]{yu-yu:mass_surface}, $m_\Pi(O_p)$ is a quadratic $\F_{p^2}$-subfield of $\Mat_2(\F_{p^2})$ in Case (II), and $m_\Pi(O_p)=\F_{p^2} \cdot \bbI_2$ in Case (I), where $\bbI_2=\begin{bmatrix}  
    \begin{smallmatrix}
        1 & 0 \\
        0 & 1
    \end{smallmatrix}\end{bmatrix}$. Thus, 
    \[ \Nr_{\F_{p^2}/\Fp} \circ \det(m_\Pi(O_p^\times))= \begin{cases}
       (\Fp^\times)^2, & \text{Case (I)}; \\
       \Fp^\times, &  \text{Case (II)}, 
   \end{cases}\]
    and Equation~\eqref{eq:19} follows from Lemma~\ref{lm:15}. 

    (2) Let $X_1$ and $X_2$ be two supersingular abelian surfaces in Case (II). By Theorem~\ref{thm:oort}, for any supersingular abelian surface $X$ with $a(X)=1$, there is an isogeny $\varphi: E_0^2 \otimes k \to X$ of degree $p$.  Let $\varphi_i: E_0^2 \otimes k \to X_i$, $i=1,2$, be two minimal isogenies of degree $p$ with respective parameter $t_i\in \F_{p^4}\setminus \F_{p^2}$. 
    To show $X_1\simeq X_2$, it suffices to show $X_1[p^\infty]\simeq X_2[p^\infty]$ by statement (1). Let $M_0$ and $M_i$ be the respective contravariant \dieu modules of $E_0^2\otimes k$ and $X_i$. One has $VM_0\subset M_i \subset M_0$ for $i=1,2$. To see this, as $\ker \varphi_i\simeq \alpha_p$ and $M(\ker \varphi_i)\simeq M_0/M_i$, we have $V(M_0/M_i)=V(M(\ker \varphi_i))=V(M(\alpha_p))=0$. Since $\varphi_i$ are minimal isogenies, any isomorphism $M_1\isoto M_2$ lifts to a unique automorphism of $M_0$. $\End_{\rm DM}(M_0)$ acts on the set of lines of $M_0/VM_0$ through the quotient $\End_{\rm DM}(M_0)\onto \Mat_2(\F_{p^2})$, and the action of $\Aut_{\rm DM}(M_0)$ factors through the action of $\GL_2(\F_{p^2})$ on $\bbP^1(k)$ by linear fractional transformations. 
    Since $\F_{p^2}(t_1)=\F_{p^2}(t_2)$, we have $t_2=(a t_1+b)/(c t_1+d)$, for some $\left (\begin{smallmatrix}
        a & b\\
        c & d 
    \end{smallmatrix} \right ) \in \GL_2(\F_{p^2}) $ and hence $\Aut_{\rm DM}(M_0)$ acts transitively on $\F_{p^4}\setminus \F_{p^2}$. This shows $M_1\simeq M_2$ and hence the proposition. \qed
\end{proof}

We see from Lemma~\ref{lm:15} that the $p$-divisible group of a supersingular abelian variety $X$ characterizes $X$ up to finitely many possibilities. Moreover, such a finite set is completely determined by $\End(X)\otimes \Zp$. Thus, the local endomorphism ring $\End(X)\otimes \Zp$ serves  as a good and convenient invariant for studying supersingular abelian varieties as well as for exploring the geometry of the moduli space $\calS_g$ of $g$-dimensional principally polarized abelian varieties. 
In some sense, Theorems~\ref{thm:oort} and \ref{thm:DOS} are just instances of such a characterization in terms of geometry.  
For $g=3$, the structure of $\End(X)\otimes \Zp$ has been explored in \cite{KYY}. The authors also describe how this invariant varies in the moduli space $\calS_3$, or more precisely in a model irreducible component $\calP'_{3,\eta}$ of $3$-dimensional rigid polarized flag type quotients (PFTQ's). 
Based on \cite{KYY}, one can further compute the size of $\Lambda_X$ for each $(X,\lambda)\in \calS_3$ using the same method of Lemma~\ref{lm:15}. 

In~\cite{IKY} the authors determine precisely when a principally polarized abelian variety is determined by its $p$-divisible group. Namely, they prove:
\begin{thm}[{\cite[Corollary B]{IKY}}]\label{thm:IKY}
   A principally polarized abelian variety $(X,\lambda)$ over $k$ is uniquely determined by its polarized $p$-divisible group if and only if $X$ is supersingular and one of the following three cases holds:
\begin{itemize}
\item [(i)] $g=1$ and $p\in \{2,3,5,7,13\}$;
\item [(ii)] $g=2$ and $p=2,3$; 
\item [(iii)] $g=3$, $p=2$ and $a(X)\ge 2$. 
\end{itemize}      
\end{thm}

By Lemma~\ref{lm:15}, one obtains the following characterization of an abelian variety by its $p$-divisible group in terms of arithmetic.

\begin{cor} An abelian variety $X$ of dimension $>1$ over $k$ is determined by its $p$-divisible group if and only if $X$ is supersingular and $\Nr(O_p^\times)=\Zp^\times$, where $O:=\End(X)$.     
\end{cor}
    
It is desirable to reveal this arithmetic condition in terms of geometry. For example, what is the locus in a parameter space $\calX\to S$ for abelian varieties defined by the condition $\Nr(O_p^\times)=\Zp^\times$? For $g=2$, we have the following result.

\begin{prop}
    Let $\calZ$ be the moduli space of isogenies $\varphi: E_0^2\otimes k\to X$ of degree $p$, and fix an identification $\calZ=\bbP^1$.
    Let $\calY\subseteq \calZ$ denote the locus consisting of $\varphi$ satisfying $\Nr(O_p^\times)=\Zp^\times$. Then 
    \[ \calY=\begin{cases}
        \bbP^1 & \text{if $p=2$;} \\
        \bbP^1(\F_{p^4}) & \text{if $p>2$.} 
    \end{cases}\]    
\end{prop}
\begin{proof}
    By Lemma~\ref{lm:15}, we have $|\Lambda_X|=1$ if and only if $\Nr(O_p^\times)=\Zp^\times$. It follows from Proposition~\ref{prop:19} that  the locus $\calY$ consists of the isogenies $\varphi$ for which  the abelian surface $X$ is  superspecial ($t\in \bbP(\F_{p^2})$), or in Case (II), or in Case (I) and $p=2$. Thus, the assertion follows. A priori, $\calY$ is a subvariety defined over $k$ but it is actually defined over $\Fp$ as it is stable under the $\Gal(\Fpbar/\Fp)$-action, following from our assertion.  \qed 
\end{proof}

Oort's conjecture \cite[Question 4]{edixhoven-moonen-oort} states that when $g\ge 2$, every geometric generic member in $\calS_g$ has automorphism group $\{\pm1 \}$.
For $g=2$ and $p>2$, Oort's conjecture has been proved by Ibukiyama~\cite{ibukiyama:ppas}, and by Karemaker and Pries \cite{karemaker-pries} independently, with a counterexample in  $p=2$. Oort's conjecture for $g=3$ and $p>2$ has been proved by Karemaker, Yobuko and the author \cite{KYY}, again with a counterexample in $p=2$. In the preprint~\cite{karemaker-yu:sseooc} Karemaker and the author prove the following result.

\begin{thm}
    Oort’s conjecture holds true for the cases where either $g=4$ and $p>2$, or $g$ is even and $p \ge 5$. Moreover, it fails for $g=4$ and $p=2$.  
\end{thm}

To disprove Oort's conjecture for $(g,p)=(4,2)$, the authors compute the automorphism group of any generic polarized $p$-divisible group, and use a result on automorphism groups of polarized abelian varieties; see~\cite[Corollary~1.5]{KTY}. 

In the preprint~\cite{dusan}, Dragutinovi{\'c} obtains a different proof of Oort's conjecture when $g=4$ and $p>2$, as well as a new proof when $g=3$ and $p>2$, using the moduli space of curves in both cases. Very recently, Oort's conjecture has been proved by Viehmann~\cite{viehmann:oort} in full generality except for a few counterexamples as above. A while later, Karemaker and the author~\cite{karemaker-yu:oortodd} announced a different proof for $p>2$.

\ \\
\npr{{\bf Acknowledgments.}} The author wishes to express his sincere 
gratitude to Professor Frans Oort for his pioneering works, from 
which the author learned a lot. He also wishes to thank Shushi Harashita and Momonari Kudo for their kind invitation to the RIMS conference ``Theory and Applications of Supersingular Curves and Supersingular Abelian Varieties II", and thanks the referee for his/her valuable comments that improved the present article significantly. The author was partially supported by the NSTC grants 112-2115-M-001-010 and 113-2115-M-001-001, and Academia Sinica IVA grant.

\end{document}